\documentclass[11pt]{amsart}
\usepackage[margin=1in]{geometry}
\numberwithin{equation}{section}
\newtheorem{thm}{Theorem}[section]
\newtheorem{prop}[thm]{Proposition}
\newtheorem{lem}[thm]{Lemma}

\newtheorem{defn}[thm]{Definition}

\newtheorem{clm}[thm]{Claim}

\newcommand{\ut}{u_\tau}
\newcommand{\pst}{v_\tau}

\newcommand{\RN}{\mathbb{R}^N}

\newcommand{\omt}{\Omega_T}
\newcommand{\io}{\int_\Omega}

\newcommand{\po}{\partial\Omega }

\newcommand{\edu}{\plap}
\newcommand{\plap}{\textup{div}\left(\rho(\nua^2)\nabla u\right)}
\newcommand{\tplap}{\textup{div}\left(\rho_\tau(\nua^2)\nabla u\right)}
\newcommand{\tplapt}{\textup{div}\left(\rho_\tau(\nuta^2)\nabla u_\tau\right)}

\newcommand{\nua}{|\nabla u|}
\newcommand{\nuta}{|\nabla \ut|}

\begin{document}
	\title[A continuum model for relaxation of interacting steps in crystal surfaces]{Mathematical validation of a continuum model for relaxation of interacting steps in crystal surfaces in $2$ space dimensions}
	\author{Xiangsheng Xu}\thanks
	{Department of Mathematics and Statistics, Mississippi State
		University, Mississippi State, MS 39762.
		{\it Email}: xxu@math.msstate.edu. {\it Calc. Var. Partial Differ. Equ.}, to appear.}
	\begin{abstract} In this paper we study the boundary value problem for the equation\\ $\mbox{div}\left(D(\nabla u)\nabla\left(\mbox{div}\left(|\nabla u|^{p-2}\nabla u+\beta\frac{\nabla u}{|\nabla u|}\right)\right)\right)+au=f$ in the $z=(x,y)$ plane. This problem is derived from a continuum model for the relaxation of a crystal surface below the roughing temperature. The mathematical challenge is of two folds. First, the mobility $D(\nabla u)$ is a $2\times 2$ matrix whose smallest eigenvalue is not bounded away from $0$ below. Second, the equation contains the $1$-Laplace operator, whose mathematical properties are still not well-understood. Existence of a weak solution is obtained. In particular,  $|\nabla u|$ is shown to be bounded when $p>\frac{4}{3}$.   
		\end{abstract}
	\maketitle

	\section{Introduction}
	
 Let $\Omega$ be a bounded domain in the $z=(x,y)$ plane with sufficiently smooth boundary $\po$. Denote 
 by $\nu$ the unit outward normal to $\po$. We consider the boundary value problem 
 \begin{eqnarray}
 -\textup{div}\left[D(\nabla u){\nabla v}\right] +a u&=&f\ \ \mbox{in $\Omega$,}\label{sta1}\\
 -\mbox{div}\left[\rho(\nua^2)\nabla u\right]&=&v\ \ \mbox{in $\Omega$},\label{sta2}\\
 \nabla u\cdot\nu&=&D(\nabla u){\nabla v}\cdot\nu=0\ \ \mbox{on $\po$,}\label{sta3}
 \end{eqnarray}	
 where $D(\nabla u)$ is a given $2\times 2$ matrix of $\nabla u$ whose eigenvalues may take the value $0$, $a\in (0,\infty)$,  $f=f(z)$ is a known function of its argument, and\begin{equation*}
 \rho(s)=s^{\frac{p-2}{2}}+\beta s^{-\frac{1}{2}}\ \ \mbox{for some $\beta>0, \ p>1$}.
 \end{equation*}  Precise assumptions on the given data will be made at a later time.
 
 Our interest in the problem originates in the mathematical description of the evolution of a crystal surface. 
 It is now well-established that the continuum relaxation of a crystal surface below the roughing temperature is governed by the conservation law 
 \begin{equation*}
 \partial_t u + \mbox{div} J = 0,
 \end{equation*}
 where $u$ is the surface height and $J$ is the adatom flux which is related to
 the mobility $D$ and the local equilibrium density of adatoms $\Gamma_s$ through Fick's law \cite{MK}. This gives 
 $$J = -D\nabla\Gamma_s.$$
 An expression for $\Gamma_s$ can be inferred from 
 the Gibbs-Thomson relation \cite{KMCS,RW,MK} to be
 $$
 \Gamma_s=\rho_0e^{\frac{\mu}{kT_s}},$$
 where $\mu$ is the chemical potential, $\rho_0$ is the constant reference density, $T_s$ is the temperature, and $k$ is the Boltzmann constant. We consider the mobility $D=D(\nabla u)$ introduced in \cite{MK}, which has the form
 \begin{equation*}
 D(\nabla u)=S\Lambda S^T,
 \end{equation*}
 where
 \begin{equation*}
 S=\frac{1}{|\nabla u|}\left(\begin{array}{ll}
 u_{x}&-u_{y}\\
 u_{y}&u_{x}
 \end{array}\right), \ \ \Lambda=\left(\begin{array}{ll}
 \frac{1}{1+q|\nabla u|}&0\\
 0&1
 \end{array}\right) \ \mbox{for some $q\geq 0$}.
 \end{equation*}
 A simple calculation shows
 \begin{equation*}
 D(\nabla u)=\left(\begin{array}{cc}
 1+\frac{u_x^2}{\nua^2}\left(\frac{1}{1+q\nua}-1\right)&\frac{u_xu_y}{\nua^2}\left(\frac{1}{1+q\nua}-1\right)\\
 \frac{u_xu_y}{\nua^2}\left(\frac{1}{1+q\nua}-1\right)&1+\frac{u_y^2}{\nua^2}\left(\frac{1}{1+q\nua}-1\right)
 \end{array}\right).
 \end{equation*}
 Crystal surfaces are known to develop facets, where $\nabla u=0$. To define $D(\nabla u)$ there, we observe that
 \begin{equation*}
 \frac{u_x^2}{\nua^2},\ \ \frac{u_xu_y}{\nua^2},\ \ \frac{u_y^2}{\nua^2}
 \end{equation*}
 are all bounded functions for $\nabla u\ne 0$, while $\frac{1}{1+q\nua}-1=0$  on the set $\{\nabla u=0\}$. 
 Thus it is natural for us to set
 \begin{equation*}
 D(\nabla u)=I,\ \mbox{the $2\times 2$ identity matrix, for $\nabla u= 0$.}
 \end{equation*} 
 That is, $D(\nabla u)$ is well-defined a.e..
 Obviously, $S$ is unitary. Hence,
 \begin{equation}\label{mell}
 D(\nabla u)\xi\cdot\xi\geq \frac{1}{1+q|\nabla u|}|\xi|^2\ \ \ \mbox{for each $\xi\in \mathbb{R}^2$}.
 \end{equation}
 
 Denote by $\Omega$
 the ``step locations area" of interest.  Then we can take the general surface energy $G(u)$ to be
 \begin{eqnarray*}
 	G(u)=\frac{1}{p}\io|\nabla u|^pdz+\beta\io|\nabla u|dz,\ \ p\geq 1,\ \ \beta\in \mathbb{R}.
 \end{eqnarray*}
 A justification for this can be found in \cite{CRSC}. As observed in \cite{MW},  this type of energy forms can retain many of the interesting features of the microscopic system that are lost in the more standard scaling regime. We shall assume that
 \begin{equation*}
 \beta >0.
 \end{equation*}
 The chemical potential $\mu$ is
 defined as the change per atom in the surface energy. That is,
 \begin{equation*}
 \mu=\frac{\delta G}{\delta u}=-\mbox{div}\left(\nua^{p-2}\nabla u+\beta\frac{\nabla u}{\nua}\right)=-\plap.
 \end{equation*}
 To give a proper meaning to the term $\frac{\nabla u}{\nua}$, we follow the approach adopted in \cite{KV} and introduce the function
 \begin{equation}\label{fpd}
 \Phi(\xi)=|\xi|, \ \ \ \xi\in \mathbb{R}^2.
 \end{equation}
 Then $\partial\Phi$, the subgradient of $\Phi$, is given by
 \begin{equation*}
 \partial\Phi(\xi)=\left\{\begin{array}{ll}
 \frac{\xi}{|\xi|} & \mbox{if $\xi\ne 0$,}  \\
 \{\xi: |\xi|\leq 1\}& \mbox{if $\xi=0$.}
 \end{array}\right.
 \end{equation*}
 We say $h=\frac{\nabla u}{\nua}$ if $h\in \left(L^\infty(\Omega)\right)^2$ and
 \begin{equation}\label{hdef}
 h(z)\in\partial\Phi(\nabla u(z))\ \ \mbox{for a.e. $z\in \Omega$.}
 \end{equation}
 After incorporating
 all the physical parameters (except $\beta$ and $q$) into the scaling of the time and/or spatial variables \cite{GLL3,LMM}, we can rewrite the
 evolution equation for $u $  as
 \begin{equation}\label{exp}
 \partial_t u=\mbox{div}\left(D(\nabla u)\nabla e^{\frac{\delta G}{\delta u}}\right).
 \end{equation}
 
 As in \cite{GK},  we linearize the exponential term
 \begin{equation*}
 e^{-\plap}\approx 1-\plap,
 \end{equation*}
 the above equations reduces to 
 \begin{equation}\label{p1}
 \partial_t u
 =-\textup{div}\left[D(\nabla u){\nabla\plap}\right].
 \end{equation}
 This equation is assumed to hold in a space-time domain  $\omt\equiv \Omega\times(0,T)$, $ T>0$,
 coupled with  the following  initial boundary conditions
 \begin{eqnarray}
 \nabla u\cdot\nu=\nabla \edu\cdot\nu &=& 0 \ \ \ \mbox{ on $\Sigma_T\equiv \partial\Omega\times(0,T)$},\label{p2}\\
 u(z,0)&=& u_0(z) \ \ \ \mbox{on $\Omega$.}\label{p3}
 \end{eqnarray}
 The rigorous mathematical analysis of nonlinear differential equations depends primarily upon deriving estimates, but typically also upon using these estimates to justify limiting procedures of various sorts. 
 Unfortunately,
 known priori estimates for this problem are rather weak. As a result, an existence theorem seems to be hopeless.  We will focus on an associated stationary problem instead. This problem is obtained by discretizing the time derivative in \eqref{p1}. Subsequently, we arrive at the following stationary equation 
 \begin{equation}\label{mobility}
 \frac{u-g}{\delta}+\textup{div}\left[D(\nabla u)\plap\right]=0\ \ \mbox{in $\Omega$.}
 \end{equation}
 Here $g$ is a given function. Initially, $g=u_0(x)$. The positive number $\delta$ is the step size. Set $a=\frac{1}{\delta}$  and $f=\frac{1}{\delta}g$. This leads to the boundary value problem \eqref{sta1}-\eqref{sta2}. 
 
 The objective of this paper is to establish an existence assertion for the stationary problem  \eqref{sta1}-\eqref{sta2}, while the time-dependent problem \eqref{p1}-\eqref{p3} is left open. We view our work here as a first step in attacking the more challenging time-dependent case.

 If $D(\nabla u)$ is the identity matrix $I$, both equations \eqref{p1} and \eqref{exp},  coupled with various types of initial boundary conditions, have received tremendous attention.
 For the former, we would like to mention \cite{GK} where the authors proved that there is a finite time extinction of solutions if $p > 1$, while in the latter case we refer the reader to \cite{LX} and the reference therein. The gradient flow theory is essential to the existence of a solution in the existing literature \cite{DFL,G1,GG,GK,GLLX,KV,LX}. If $p\ne 2$ in \eqref{exp} or $D(\nabla u)\ne I$ in \eqref{exp} or \eqref{p1}, the resulting equations have received much less consideration. The gradient flow theory does not seem to be as effective here.  In \cite{G1}, the author dealt with a non-constant, singular $D(\nabla u)$. However, the $p$-Laplace operator in the exponent in \eqref{exp} had been modified there so that the resulting equation became a gradient flow. Physically, one takes $D(\nabla u)=I$ in the diffusion-limited regime of crystal surfaces where the dynamic is dominated by the diffusion across the terraces. However, if the attachment and detachment of atoms at step edges are the main focus, the mobility $D(\nabla u)$ can take very complicated forms \cite{GLL,XX}.

 We return to the stationary problem  \eqref{sta1}-\eqref{sta3}. We give the following definition of a weak solution for the problem.
 \begin{defn}
 	We say that a triplet $(u,v,h)$, where $h$ is a vector $(h_1,h_2)^T$, is a weak solution to \eqref{sta1}-\eqref{sta2} if the following conditions hold:
 	\begin{enumerate}
 		\item[\textup{(D1)}]
 		$u\in W^{1,p}(\Omega)$,
 		$v\in W^{1,\frac{2p}{p+1}}(\Omega)$,  and $h\in \left(L^\infty(\Omega)\right)^2$ with $h(z)\in\partial\Phi(\nabla u(z))$ for a.e. $z\in \Omega$, where $\Phi$ is given as in \eqref{fpd}.
 		\item[\textup{(D2)}] The functions $u, h$ satisfy
 		the
 		problem
 		\begin{eqnarray}
 		-\textup{div}\left(|\nabla u|^{p-2}\nabla u+\beta h\right)&=&v\ \ \textup{in $\Omega$,}\label{rr1}\\
 		\left(|\nabla u|^{p-2}\nabla u+\beta h\right)\cdot\nu&=&0 \ \ \textup{on $\po$}\nonumber
 		\end{eqnarray}
 		in the weak sense, i.e.,
 		\begin{equation}
 		\io\left(|\nabla u|^{p-2}\nabla u+\beta h\right)\cdot\nabla\varphi dz=\io v\varphi dz\ \ \ \mbox{for each $\varphi\in  W^{1,p}(\Omega)\cap L^\infty(\Omega)$, }
 		\end{equation}
 		while $v$ is a weak solution of the problem
 		\begin{eqnarray*}
 			-\textup{div}\left(D(\nabla u)\nabla v\right)+au=f\ \ \textup{in $\Omega$,}\\
 			D(\nabla u)\nabla v\cdot\nu&=&0 \ \ \textup{on $\po$.}
 		\end{eqnarray*}
 	\end{enumerate}
 \end{defn}

 
 Our main result is the following 
 \begin{thm}\label{th1.1}	Let $\Omega$ be a bounded domain in $\mathbb{R}^2$ with Lipschitz boundary $\po$ and $a$ and $\beta$ two positive numbers. Assume that
 	\begin{eqnarray}
 	1&<&p\leq 2\ \ \mbox{and}\label{r33}\\
 	f&\in&  W^{1,p}(\Omega)\cap L^\infty(\Omega).
 	\end{eqnarray}
 	Then
 	there is a  
 	weak solution $u$ to \eqref{sta1}-\eqref{sta2}. 
 	If, in addition,
 	\begin{equation}\label{r1111}
 	p>\frac{4}{3}\ \ \mbox{and $\po$ is $C^{1,1}$,} 
 	\end{equation}
 	then we have $\nua\in L^\infty(\Omega)$.
 \end{thm}

 This theorem is not trivial even when $p=2$. Indeed, in this case we can represent \eqref{rr1} as
 \begin{equation*}
 -\Delta u=v+\beta\mbox{div}h\ \ \textup{in $\Omega$.}
 \end{equation*} 
 Remember that we only have $h\in (L^\infty(\Omega))^2$. Thus the classical $W^{1,q} $ estimate (\cite{R}, p. 82 ) can only yield $u\in W^{1,q}(\Omega)$ for each $q> 1$.
 Although our proof will be carried out under the  assumption \eqref{r33}
 we believe that our theorem is still valid for $p> 2$. In fact, the existence part of the proof in this case only needs mild modification of that for Theorem \ref{th1.1}.  We will leave the details to the interested reader.
 The uniqueness assertion for problem \eqref{sta1}-\eqref{sta2} is still open. The difficulty here is due to the fact that the operator
 $\textup{div}\left[D(\nabla u)\nabla\left({\plap}\right)\right]$ does not seem to be monotone.
 
 The $1$-Laplace operator, denoted by $\Delta_1$,  is the so-called mean curvature operator. It has the property
 \begin{equation*}
 \Delta_1\varphi(u)=\Delta_1u \ \ \mbox{for each smooth increasing function $\varphi$ in one variable}.
 \end{equation*}
 Regularity properties of $1$-harmonic functions are still not well-understood \cite{GG}.
 The redeeming feature in our problem is that we also have a $p$-Laplace operator with $p>1$. Our analysis reveals that this $p$-Laplace operator can dominate the $1$-Laplace operator in a lot of aspects. Nonetheless, many techniques employed in the study of $p$-harmonic functions are no longer applicable to the $p$-Poisson equation. One reason for this is that one can remove the singular term $|\nabla u|^{p-2}$ from the $p$-Laplace equation. To see this, we carry out the divergence in the equation, divide through the resulting equation by $|\nabla u|^{p-2}$, and thereby obtain
 $$
 \left( I+\frac{p-2}{|\nabla u|^2}\nabla u\otimes\nabla u \right):\nabla^2 u=0,
 $$
 where $\nabla^2u$ denotes the Hessian of $u$. Moreover, the notations
 \begin{eqnarray*}
 	\xi\otimes\eta&=&\xi\eta^T \ \ \mbox{for $ \xi,\eta\in \mathbb{R}^N$},\\
 	A:B&=&\sum_{i,j=1}^{N}a_{ij}b_{ij}\ \  \mbox{ for $A, B \in \mathcal{M}^{N\times N}$, the space of all $N\times N$ matrices, } 
 \end{eqnarray*}
 have been employed.
 
 Note that the coefficient matrix in the above equation is uniformly elliptic.
 Obviously, this can not be done for the $p$-Poisson equation. In fact, this largely accounts for our assumption $p>\frac{4}{3}$. To establish an upper bound for $\nua$, we derive an equation satisfied by $u_x$ (resp. $u_y$). Unlike the case of $p$-harmonic functions \cite{L}, the equation for $u_x$ is no longer uniformly elliptic. We circumvent this problem by suitably modifying the classical De Giorgi technique \cite{D}. Remember that an estimate of Caccioppoli-type does not hold for the $1$-Laplace operator. Thus it is a little bit surprising that we are still able to obtain the boundedness of $|\nabla u|$.

 A solution to \eqref{sta1}-\eqref{sta2} will be constructed as a limit of a sequence of approximate solutions. Roughly, we regularize the problem by replacing $\nua$ with $(\nua^2+\tau)^{\frac{1}{2}}$ for $\tau\in (0,1]$. 
 
 This work is organized as follows. In Section \ref{sec2} we collect some relevant known results. The existence part in Theorem \ref{th1.1} is established in Section \ref{sec3}, while the boundedness of $|\nabla u|$ is proved in Section \ref{sec4}. Finally, we make some remarks about the notation. The letter $c$ denotes a positive constant. In theory, its value can be computed from various given data.

 \section{Preliminaries}\label{sec2}
 In this section we state a few preparatory lemmas.

 Relevant interpolation inequalities for Sobolev spaces are listed in the following lemma.
 \begin{lem}\label{linterp}	Let $\Omega$ be a bounded domain in $\mathbb{R}^N$. Denote by
 	$\|\cdot\|_p$ the norm in the space $L^p(\Omega)$. Then we have:
 	\begin{enumerate}
 		\item $ \|f\|_s\leq\varepsilon\|f\|_r+\varepsilon^{-\sigma} \|f\|_p$, where
 		$\varepsilon>0, p\leq s\leq r$, and $\sigma=\left(\frac{1}{p}-\frac{1}{s}\right)/\left(\frac{1}{s}-\frac{1}{r}\right)$;
 		\item If $\po$ is Lipschitz, then for each $\varepsilon >0$ and each $s\in (1, p^*)$, where $p^*=\frac{pN}{N-p}$ if $N>p\geq 1$ and any number bigger than $p$ if $N=p$, there is a positive number $c=c(\varepsilon, p, \po)$ such that
 		\begin{eqnarray}
 		\|f\|_s&\leq &\varepsilon\|\nabla f\|_p+c\|f\|_1\ \ \mbox{for all $f\in
 			W^{1,p}(\Omega)$}.\label{otn9}
 		\end{eqnarray}
 		If $s=p^*$, then we have
 		\begin{equation}
 		\|f\|_{p^*}\leq c(\|\nabla f\|_p+\|f\|_1)\ \ \mbox{for all $f\in
 			W^{1,p}(\Omega)$}.\label{otn19}
 		\end{equation}
 		Here $c$ depends on $p, \po$.
 	\end{enumerate}
 \end{lem}
 This lemma is largely contained in Chap. II of \cite{LSU}.
 One can also prove \eqref{otn9} and \eqref{otn19} by a contradiction argument (\cite{GT}, p.174).
 
 We will collect a few frequently used elementary inequalities in the following three lemmas.
 \begin{lem}\label{elmen}Assume that $a,b$ are two positive numbers. Then we have:
 	\begin{equation*}
 	ab\leq \varepsilon a^p+\frac{1}{\varepsilon^{1/(p-1)}}b^{p^{\prime}},
 	\end{equation*}
 	where $\varepsilon>0,\, p>1$, and $p^\prime=\frac{p}{p-1}$. 
 \end{lem}
 This lemma is the so-called Young's inequality (\cite{LSU}, p. 58). 
 \begin{lem}\label{lplap}Let $x,y$ be any two vectors in $\RN$. Then:
 	\begin{enumerate}
 		\item[\textup{(i)}] For $p\geq 2$,
 		\begin{equation*}
 		\left(\left(|x|^{p-2}x-|y|^{p-2}y\right)\cdot(x-y)\right)\geq \frac{1}{2^{p-1}}|x-y|^{p};
 		\end{equation*}
 		\item[\textup{(ii)}] For $1<p\leq 2$,
 		\begin{equation*}
 		\left(1+|x|^2+|y|^2\right)^{\frac{2-p}{2}}\left(\left(|x|^{p-2}x-|y|^{p-2}y\right)\cdot(x-y)\right)\geq (p-1)|x-y|^2.
 		\end{equation*}
 	\end{enumerate}
 \end{lem}
 The proof of this lemma is contained in (\cite{LI}, p. 71-74).

 \begin{lem}\label{ynb}
 	Let $\{y_n\}, n=0,1,2,\cdots$, be a sequence of positive numbers satisfying the recursive inequalities
 	\begin{equation*}
 	y_{n+1}\leq cb^ny_n^{1+\alpha}\ \ \mbox{for some $b>1, c, \alpha\in (0,\infty)$.}
 	\end{equation*}
 	If
 	\begin{equation*}
 	y_0\leq c^{-\frac{1}{\alpha}}b^{-\frac{1}{\alpha^2}},
 	\end{equation*}
 	then $\lim_{n\rightarrow\infty}y_n=0$.
 \end{lem}
 This lemma can be found in (\cite{D}, p.12).

 Our existence theorem is based upon the following fixed point theorem, which is often called the Leray-Schauder Theorem (\cite{GT}, p.280).
 \begin{lem}
 	Let $B$ be a map from a Banach space $\mathcal{B}$ into itself. Assume:
 	\begin{enumerate}
 		\item[(H1)] $B$ is continuous;
 		\item[(H2)] the images of bounded sets of $B$ are precompact;
 		\item[(H3)] there exists a constant $c$ such that
 		$$\|z\|_{\mathcal{B}}\leq c$$
 		for all $z\in\mathcal{B}$ and $\sigma\in[0,1]$ satisfying $z=\sigma B(z)$.
 	\end{enumerate}
 	Then $B$ has a fixed point.
 \end{lem}

 \section{Existence}\label{sec3}
 In this section we first design an approximation scheme for problem \eqref{sta1}-\eqref{sta2}. Then we obtain a weak solution by passing to the limit in our approximate problems.
 
 As discussed in the introduction, we let
 \begin{equation}
 v=-\plap.
 \end{equation}
 Then regularize this equation by adding the term $\tau |u|^{p-2}u,\ \tau\in (0,1]$, to its right-hand side and replacing $\rho$ by 
 \begin{equation}\label{r2}
 \rho_\tau(s)=(s+\tau)^{\frac{p-2}{2}}+\beta (s+\tau)^{-\frac{1}{2}}.
 \end{equation}
 The former is due to the Neumann boundary condition in our problem, while the latter takes care of the problem with the set where  $\nua=0$. For the same reason, we substitute $ D(\nabla u)$ with 
 \begin{equation}
 D_\tau(\nabla u)=
 \left(\begin{array}{cc}
 1+\tau+\frac{u_x^2}{\nua^2+\tau}\left(\frac{1}{1+q\nua}-1\right)&\frac{u_xu_y}{\nua^2+\tau}\left(\frac{1}{1+q\nua}-1\right)\\
 \frac{u_xu_y}{\nua^2+\tau}\left(\frac{1}{1+q\nua}-1\right)&1+\tau+\frac{u_y^2}{\nua^2+\tau}\left(\frac{1}{1+q\nua}-1\right)
 \end{array}\right).
 \end{equation} 
 It is easy to verify that we have
 \begin{eqnarray}
 D_\tau(\nabla u)\xi\cdot\xi&=&(1+\tau)|\xi|^2+\left(\frac{1}{1+q\nua}-1\right)\frac{(\nabla u\cdot\xi)^2}{\nua^2+\tau}\nonumber\\
 &\geq&\left( \frac{1}{1+q\nua}+\tau\right)|\xi|^2\ \ \mbox{for each $\xi\in \mathbb{R}^2$.}\label{r22}
 \end{eqnarray}
 Furthermore, each entry in $D_\tau(\nabla u)$ is  bounded by $2$.
 Finally,  we still need to add $\tau v$ to \eqref{sta1}.
 This leads to the study of the system 
 \begin{eqnarray}
 -\mbox{div}\left(D_\tau(\nabla u)\nabla v\right)+\tau v&=&f-au\ \ \ \mbox{in $\Omega$},\label{ot1}\\
 -\tplap +\tau|u|^{p-2} u &=&v \ \ \ \mbox{in $\Omega$}\label{plap11}
 \end{eqnarray}
 coupled with the boundary conditions
 \begin{equation}
 \rho_\tau(|\nabla u|^2)\nabla u\cdot\nu=D_\tau(\nabla u)\nabla v\cdot\nu=0\ \ \ \mbox{on $\partial\Omega$}.\label{ot2}
 \end{equation}
 This is our approximating problem.
 Basically, we have transformed a fourth-order equation into a system of two second-order equations.

 \begin{thm}\label{p21}
 	Let $\Omega$ be a bounded domain in $\mathbb{R}^2$ with  Lipschitz boundary, and assume that $1<p$ and $f\in L^\infty(\Omega)$. Then there is a weak solution $(v, u)$ to \eqref{ot1}-\eqref{ot2} with 
 	\begin{eqnarray*}
 		v\in W^{1,2}(\Omega)\cap L^\infty(\Omega),\ \ \ 	u\in W^{1,p}(\Omega)\cap L^\infty(\Omega).
 	\end{eqnarray*}.
 \end{thm}
 \begin{proof}
 	The existence assertion will be established via the Leray-Schauder Theorem. For this purpose, we let
 	\begin{equation}
 	q=\max\left\{\frac{p}{p-1},2\right\}
 	\end{equation}
 	and define an operator $B$ from $L^q(\Omega)$
 	into itself as follows: for each $g\in L^q(\Omega)$ we say $B(g)= v$ if $ v$ is the unique solution of the linear boundary value problem
 	\begin{eqnarray}
 	-\mbox{div}\left(D_\tau(\nabla u)\nabla v\right)+\tau v &=&f-au\ \ \ \mbox{in $\Omega$},\label{om3}\\
 	D_\tau(\nabla u)\nabla v\cdot\nu&=&0\ \ \ \mbox{on $\partial\Omega$},\label{om4}
 	\end{eqnarray}
 	where $u$ solves the problem
 	\begin{eqnarray}
 	-\tplap +\tau |u|^{p-2}u &=&g\ \ \ \mbox{in $\Omega$},\label{om5}\\
 	\rho_\tau(|\nabla u|^2)\nabla u\cdot\nu&=& 0\ \ \ \mbox{on $\partial\Omega$}.\label{om6}
 	\end{eqnarray}
 	To see that $B$ is well-defined, we can easily infer from a theorem in (\cite{O}, p.124) that the problem \eqref{om5}-\eqref{om6} has a weak solution $u$ in the space $W^{1,p}(\Omega)$. It is elementary to check that the function $\frac{1}{p}(s^2+\tau)^{\frac{p}{2}}+\beta(s^2+\tau)^{\frac{1}{2}}$ is convex on the interval $(-\infty,\infty)$, from whence follows that the function
 	\begin{equation}\label{pht}
 	\Psi_\tau(\xi)=\frac{1}{p}(|\xi|^2+\tau)^{\frac{p}{2}}+\beta(|\xi|^2+\tau)^{\frac{1}{2}}
 	\end{equation}  is convex on $\mathbb{R}^2$, and hence its gradient $\nabla \Psi_\tau(\xi)= (|\xi|^2+\tau)^{\frac{p-2}{2}}\xi+\beta (|\xi|^2+\tau)^{-\frac{1}{2}}\xi=\rho_\tau(|\xi|^2)\xi$ is monotone, i.e.,
 	\begin{equation}
 	(\rho_\tau(|\xi|^2)\xi-\rho_\tau(|\eta|^2)\eta)\cdot(\xi-\eta)\geq 0\ \ \ \mbox{for all $\xi,\eta\in\mathbb{R}^2$.}
 	\end{equation} This together with the fact that $|u|^{p-2}u$ is a strictly monotone function of $u$ implies that the problem \eqref{om5}-\eqref{om6} has a unique weak solution in $W^{1,p}(\Omega)$. Since $q\geq2$, we are in a position to apply the proof of Lemma 2.4 in \cite{X5}, thereby obtaining
 	\begin{equation}
 	\|u\|_\infty\leq c.
 	\end{equation}
 	Here $c$ depends on $\tau$ and $\|g\|_q$.
 	
 	Note that \eqref{om3} is a uniformly elliptic linear equation in $ v$. Existence and uniqueness of a weak solution $ v$ to the problem \eqref{om3}-\eqref{om4} in the space $W^{1,2}(\Omega)$ follow easily. Moreover, the right-hand side function $f-au$ lies in $L^\infty(\Omega)$. Subsequently, we also have
 	$ v\in L^\infty(\Omega)$. That is, the range of $B$ is contained in 
 	$W^{1,2}(\Omega)\cap L^\infty(\Omega)$. The Sobolev embedding theorem asserts that this function space is compactly embedded in $L^q(\Omega)$. This immediately implies that $B$ maps bounded sets into precompact ones.
 	It is fairly straightforward to show that $B$ is also continuous.
 	
 	It remains to verify
 	(H3) in the Leray-Schauder Theorem. That is, we must 
 	show that there is a positive number $c$ such that
 	\begin{equation}
 	\| v\|_q\leq c\label{ot8}
 	\end{equation}
 	for all $ v\in L^q(\Omega)$ and $\sigma\in [0,1]$ satisfying
 	$$ v=\sigma B( v).$$
 	This equation is equivalent to the boundary value problem
 	\begin{eqnarray}
 	-\mbox{div}\left(D_\tau(\nabla u)\nabla v\right)+\tau v &=&\sigma(f-au)\ \ \ \mbox{in $\Omega$},\label{ot9}\\
 	-\tplap +\tau|u|^{p-2} u &=& v \ \ \ \mbox{in $\Omega$},\label{ot10}\\
 	\nabla u\cdot\nu=\nabla v\cdot\nu&=& 0\ \ \ \mbox{on $\partial\Omega$}.
 	\end{eqnarray}
 	To establish \eqref{ot8}, we calculate from \eqref{ot10} that
 	\begin{equation}
 	\io u v dz=\io \rho_\tau(\nua^2)\nua^2dz+\tau\io| u|^p dz\geq 0.
 	\end{equation}
 	With this in mind, we derive from \eqref{ot9} that
 	\begin{eqnarray*}
 		\io D_\tau(\nabla u)\nabla v\cdot \nabla v dz+\tau\io| v|^2dz&=&\sigma\io f v dz-\sigma a\io u v dz\nonumber\\
 		&\leq &\io |f v|dz\leq \frac{\tau}{2}\io| v|^2dz+\frac{c}{\tau}\io|f|^2dz.
 	\end{eqnarray*}
 	By \eqref{r22},
 	\begin{equation}
 	\io\frac{1}{1+q\nua}|\nabla v|^2dz+\tau\| v\|_2^2\leq \frac{c}{\tau}\|f\|_2^2.
 	\end{equation}
 	That is, $ v$ is bounded in $L^2(\Omega)$. This enables us to use the proof of Lemma 2.4 in \cite{X5} again to yield
 	the boundedness of $u$.
 	
 	For each $s>2$ the function $| v|^{s-2}  v$ lies in $W^{1,2}(\Omega)$ and $\nabla\left(| v|^{s-2}  v\right)=(s-1)| v|^{s-2}\nabla v$. 
 	Use	this function as a test in \eqref{ot9}
 	to obtain
 	\begin{eqnarray*}
 		(s-1)\int_\Omega| v|^{s-2}D_\tau(\nabla u)\nabla v\cdot\nabla v\, dz
 		+\tau\int_\Omega| v|^{s}\, dz&=&\sigma\int_\Omega(f-au)| v|^{s-2}  v \, dz\\
 		&\leq &\int_\Omega|f-au|| v|^{s-1}\, dz\\
 		&\leq &\|f-au\|_{s}\| v\|_{s}^{s-1}.
 	\end{eqnarray*}
 	Dropping the first integral in the above inequality yields
 	\begin{eqnarray}
 	\| v\|_{s}&\leq&\frac{1}{\tau}\|f-au\|_{s}\ \ \textup{for each $s>2$, and thus}\label{c21} \\
 	\| v\|_\infty&\leq&\frac{1}{\tau}\|f-au\|_\infty.\label{ot12}
 	\end{eqnarray}
 	This completes the proof of the theorem.
 \end{proof}	
 It is possible for us to obtain higher regularity for $(v,u)$. For example, we can easily show that $v$ is H\"{o}lder continuous on $\overline{\Omega}$. We will not pursue this here.
 \begin{proof}[Proof of Theorem \ref{th1.1}]
 	We shall show that we can take $\tau\rightarrow 0$ in \eqref{ot1}-\eqref{ot2}. For this purpose we need to derive estimates that are uniform in $\tau$. We write
 	\begin{equation}
 	u=\ut,\ \  v=\pst.
 	\end{equation}
 	Then problem \eqref{ot1}-\eqref{ot2} becomes
 	\begin{eqnarray}
 	-\textup{div}(D_\tau(\nabla \ut)\nabla\pst) +\tau\pst+a \ut &=&f\ \ \ \mbox{in $\Omega$},\label{ot1t}\\
 	-\tplapt+\tau|\ut|^{p-2} \ut &=&\pst \ \ \ \mbox{in $\Omega$},\label{ot3t}\\
 	\nabla \ut\cdot\nu&=&\nabla \pst\cdot\nu=0\ \ \ \mbox{on $\partial\Omega$}.\label{ot2t}
 	\end{eqnarray}
 	We also view $\{\ut,\pst\}$ as a sequence in the subsequent proof. Take $\tau=\frac{1}{j}$, where $j$ is a positive integer, for example. The rest of the proof is divided into several claims. 
 \end{proof}
 \begin{clm}We have
 	\begin{eqnarray}
 	\io  \frac{1}{1+q\nuta}|\nabla\pst|^2dz+\tau\io\pst^2dz+\io\Psi_\tau(\nabla\ut)dz+\tau\io|\ut|^pdz&\leq &c,\label{rue}\\
 	\|\ut\|_{W^{1,p}(\Omega)}&\leq &c,\label{ue}\\
 	\|\pst\|_{W^{1,\frac{2p}{p+1}}(\Omega)}&\leq & c,\label{ptb}
 	\end{eqnarray}
 	where the function $\Psi_\tau$ is given as in \eqref{pht}.
 \end{clm}
 \begin{proof}
 	Use $\pst$ as a test function in \eqref{ot1t} to obtain
 	\begin{equation}\label{nf1}
 	\io D_\tau(\nabla\ut)\nabla\pst\cdot\nabla\pst dz+\tau\io\pst^2dz+a\io\ut\pst dz=\io f\pst dz.
 	\end{equation}
 	With the aid of \eqref{ot3t}, we evaluate the last integral on the left-hand side in the above equation as follows:
 	\begin{eqnarray}
 	\io\ut\pst dz&=&\io(\nuta^2+\tau)^{\frac{p-2}{2}}|\nabla\ut|^2dz+\beta\io\frac{\nuta^2}{(\nuta^2+\tau)^{\frac{1}{2}}}dz+\tau\io|\ut|^pdz\nonumber\\
 	&=&\io(\nuta^2+\tau)^{\frac{p}{2}}dz+\beta\io(\nuta^2+\tau)^{\frac{1}{2}}dz+\tau\io|\ut|^pdz\nonumber\\
 	&&-\io\tau(\nuta^2+\tau)^{\frac{p-2}{2}}dz-\beta\io\frac{\tau}{(\nuta^2+\tau)^{\frac{1}{2}}}dz\nonumber\\
 	&\geq &\io(\nuta^2+\tau)^{\frac{p}{2}}dz+\beta\io(\nuta^2+\tau)^{\frac{1}{2}}dz+\tau\io|\ut|^pdz-(\tau^{\frac{p}{2}}+\beta\tau^{\frac{1}{2}})|\Omega|.\label{nf2}
 	\end{eqnarray}
 	The last step is due to \eqref{r33}. As for the right-hand side term in \eqref{nf1}, we have
 	\begin{eqnarray}
 	\io f\pst dz&=&\io(\nuta^2+\tau)^{\frac{p-2}{2}}\nabla\ut\nabla fdz+\beta\io\frac{\nabla\ut\nabla f}{(\nuta^2+\tau)^{\frac{1}{2}}}dz+\tau\io |\ut|^{p-2} \ut f dz\nonumber\\
 	&\leq &\io(\nuta^2+\tau)^{\frac{p-1}{2}}|\nabla f|dz+\beta\|\nabla f\|_1+\tau\|f\|_p\|\ut\|_p^{p-1}\nonumber\\
 	&\leq &\|\nabla f\|_p\|(\nuta^2+\tau)^{\frac{1}{2}}\|_p^{p-1}+\beta\|\nabla f\|_1+\tau\|f\|_p\|\ut\|_p^{p-1}.\label{nf3}
 	\end{eqnarray}
 	Recall from \eqref{r22} that
 	\begin{equation}
 	D_\tau(\nabla\ut)\nabla\pst\cdot\nabla\pst\geq \frac{1}{1+q\nuta}|\nabla\pst|^2.
 	\end{equation}
 	Plug the above inequality, \eqref{nf2}, and \eqref{nf3} into \eqref{nf1}, apply Lemma \ref{elmen} appropriately in the resulting inequality, remember
 	\begin{equation}
 	\tau\leq 1,
 	\end{equation}  and thereby obtain \eqref{rue}.
 	
 	Integrate \eqref{ot1t} over $\Omega$ and use \eqref{rue} to yield
 	\begin{equation}
 	\left|a\io \ut dz\right|=\left|\io fdz-\tau\io\pst dz\right|\leq c.
 	\end{equation}
 	Subsequently, we can apply the Poincar\'{e} inequality to get
 	\begin{eqnarray}
 	\|\ut\|_{p^*}&\leq &\left\|\ut-\frac{1}{|\Omega|}\io \ut dz\right\|_{p^*}+\frac{1}{|\Omega|^{1-\frac{1}{p}}}\left|\io \ut dz\right|\nonumber\\
 	&\leq &c\|\nabla\ut\|_{p}+\frac{1}{|\Omega|^{1-\frac{1}{p}}}\left|\io \ut dz\right|\leq c.\label{nf5}
 	\end{eqnarray}
 	Thus \eqref{ue} follows. 
 	
 	To see \eqref{ptb}, we integrate \eqref{ot10} to get
 	\begin{equation}
 	\left|\io\pst dz\right|=\tau\left|\io|\ut|^{p-2}\ut dz\right|\leq c.
 	\end{equation}
 	We estimate from  Poincar\'{e}'s inequality that
 	\begin{eqnarray}
 	\|\pst\|_{2p}&\leq &\left\|\pst-\frac{1}{|\Omega|}\io\pst dz\right\|_{2p}+c\nonumber\\
 	&\leq &c\left(\io|\nabla\pst|^{\frac{2p}{1+p}}dz\right)^{\frac{1+p}{2p}}+c\nonumber\\
 	&=&c\left(\io(1+q\nuta)^{\frac{p}{1+p}}\frac{1}{(1+q\nuta)^{\frac{p}{1+p}}}|\nabla\pst|^{\frac{2p}{1+p}}dz\right)^{\frac{1+p}{2p}}+c\nonumber\\
 	&\leq &c\|1+q\nuta\|_p^{\frac{1}{2}}\left(\io\frac{1}{1+q\nuta}|\nabla\pst|^{2}dz\right)^{\frac{1}{2}}+c\leq c.\label{nf4}
 	\end{eqnarray}
 	The proof is complete.
 \end{proof}
 \begin{clm}\label{utb} We have
 	\begin{equation*}
 	\|\ut\|_\infty\leq c.
 	\end{equation*}
 	
 \end{clm}
 Observe that this claim is a consequence of the Sobolev embedding theorem if $p>2$.
 \begin{proof}
 	Without loss of generality, we assume
 	\begin{equation*}
 	\|\ut\|_\infty=\|\ut^+\|_\infty.
 	\end{equation*}
 	Let $K>0$ be selected as below. Define
 	\begin{equation}\label{utbk}
 	K_n=K-\frac{K}{2^n},\ \ \ n=0,1,2,\cdots.
 	\end{equation}	
 	We use $(\ut-K_{n+1})^+$ as a test function in \eqref{ot3t} to obtain
 	\begin{equation}\label{ell5}
 	\io \rho_\tau(\nuta^2)\nabla \ut\cdot\nabla (\ut-K_{n+1})^+dz+\tau\io |\ut|^{p-2}\ut(\ut-K_{n+1})^+dz=\io \pst(\ut-K_{n+1})^+dz.
 	\end{equation}
 	Note that
 	\begin{eqnarray*}
 		|\ut|^{p-2}\ut(\ut-K_{n+1})^+&\geq&\left[(\ut-K_{n+1})^+\right]^p,\\
 		\rho_\tau(\nuta^2)\nabla \ut\cdot\nabla (\ut-K_{n+1})^+&=&\rho_\tau( \nuta^2)|\nabla (\ut-K_{n+1})^+|^2.
 	\end{eqnarray*}
 	Therefore, 
 	\begin{equation}\label{r1}
 	\io \rho_\tau(\nuta^2)|\nabla (\ut-K_{n+1})^+|^2dz\leq \io \pst(\ut-K_{n+1})^+dz
 	\end{equation}
 	Set
 	\begin{equation*}
 	U_n=\{\ut\geq K_n\}.
 	\end{equation*}
 	Remember that $p<2$. With this in mind, we can derive from \eqref{r2} and \eqref{r1} that
 	\begin{eqnarray*}
 		\io|\nabla(\ut-K_{n+1})^+|^pdz&=&\int_{U_{n+1}}|\nabla(\ut-K_{n+1})^+|^pdz\nonumber\\
 		&\leq &\int_{U_{n+1}}(|\nabla(\ut-K_{n+1})^+|^2+\tau)^{\frac{p}{2}}dz\nonumber\\
 		&=&\int_{U_{n+1}}(|\nabla(\ut-K_{n+1})^+|^2+\tau)^{\frac{p}{2}-1}\left(|\nabla(\ut-K_{n+1})^+|^2+\tau\right)dz\nonumber\\
 		&=&\int_{U_{n+1}}(|\nabla(\ut-K_{n+1})^+|^2+\tau)^{\frac{p-2}{2}}|\nabla(\ut-K_{n+1})^+|^2dz\nonumber\\
 		&&+\tau\int_{	U_{n+1}}(|\nabla(\ut-K_{n+1})^+|^2+\tau)^{\frac{p-2}{2}}dz\nonumber\\
 		&\leq &
 		\int_{U_{n+1}}\rho_\tau(\nuta^2)|\nabla (\ut-K_{n+1})^+|^2dz+\tau^{\frac{p}{2}}|U_{n+1}|\nonumber
 		\\
 		&\leq &\int_{U_{n+1}} \pst(\ut-K_{n+1})^+dz+\tau^{\frac{p}{2}}|U_{n+1}|.
 	\end{eqnarray*}
 	We estimate from \eqref{otn19} and \eqref{ptb} that
 	\begin{eqnarray*}
 		\io \pst(\ut-K_{n+1})^+dz&\leq &\|(\ut-K_{n+1})^+\|_{\frac{2p}{2-p}}\left(\int_{	U_{n+1}}|\pst|^{\frac{2p}{3p-2}}dz\right)^{\frac{3p-2}{2p}}\nonumber\\
 		&\leq &c(\|\nabla(\ut-K_{n+1})^+\|_p+\|(\ut-K_{n+1})^+\|_1)\left(\int_{	U_{n+1}}|\pst|^{\frac{2p}{3p-2}}dz\right)^{\frac{3p-2}{2p}}\nonumber\\
 		&\leq & \frac{1}{2}\io|\nabla(\ut-K_{n+1})^+|^pdz+c\left(\int_{	U_{n+1}}|\pst|^{\frac{2p}{3p-2}}dz\right)^{\frac{3p-2}{2(p-1)}}\nonumber\\
 		&&+c\|(\ut-K_{n+1})^+\|_1\nonumber\\
 		&\leq &\frac{1}{2}\io|\nabla(\ut-K_{n+1})^+|^pdz+c|U_{n+1}|^{\frac{3}{2}}+c\|(\ut-K_{n+1})^+\|_1.
 	\end{eqnarray*}
 	Combining the preceding two estimates yields
 	\begin{equation*}
 	\io|\nabla(\ut-K_{n+1})^+|^pdz\leq c|U_{n+1}|+c\|(\ut-K_{n+1})^+\|_1.
 	\end{equation*}
 	Set
 	\begin{equation*}
 	Y_n=\|(\ut-K_{n})^+\|_1.
 	\end{equation*}
 	We derive from the Sobolev embedding theorem that
 	\begin{eqnarray}
 	Y_{n+1}&\leq &\|(\ut-K_{n+1})^+\|_{\frac{2p}{2-p}}|U_{n+1}|^{\frac{3p-2}{2p}}\nonumber\\
 	&\leq &\left\|(\ut-K_{n+1})^+-\frac{1}{|\Omega|}\io(\ut-K_{n+1})^+dz\right\|_{\frac{2p}{2-p}}|U_{n+1}|^{\frac{3p-2}{2p}}+cY_{n+1}|U_{n+1}|^{\frac{3p-2}{2p}}\nonumber\\
 	&\leq&c\|\nabla(\ut-K_{n+1})^+\|_p|U_{n+1}|^{\frac{3p-2}{2p}}+cY_{n+1}|U_{n+1}|^{\frac{3p-2}{2p}}\nonumber\\
 	&\leq&c|U_{n+1}|^{\frac{3}{2}}+cY_{n+1}^{\frac{1}{p}}|U_{n+1}|^{\frac{3p-2}{2p}}+cY_{n+1}|U_{n+1}|^{\frac{3p-2}{2p}}\nonumber\\
 	&\leq&c|U_{n+1}|^{\frac{3}{2}}+cY_{n+1}^{\frac{1}{p}}|U_{n+1}|^{\frac{3p-2}{2p}}+cY_{n+1}|U_{n+1}|^{\frac{1}{2}}\nonumber\\
 	&\leq&c|U_{n+1}|^{\frac{3}{2}}+cY_{n}^{\frac{1}{p}}|U_{n+1}|^{\frac{3p-2}{2p}}+cY_{n}|U_{n+1}|^{\frac{1}{2}}.\label{utb3}
 	\end{eqnarray}
 	The last step is due to the fact that the sequence $\{Y_{n}\}$ is decreasing.
 	Note that
 	\begin{equation*}
 	Y_n=\io(\ut-K_{n})^+dz\geq(K_{n+1}-K_n)|U_{n+1}|=\frac{K}{2^{n+1}}|U_{n+1}|.
 	\end{equation*}
 	With this in mind, we derive from \eqref{utb3} that
 	\begin{equation*}
 	Y_{n+1}\leq \frac{c\sqrt{2}^{3n}}{g(K)}Y_n^{\frac{3}{2}},
 	\end{equation*}
 	where $g(K)=\min\{K^{\frac{3}{2}}, K^{\frac{3p-2}{2p}}, K^{\frac{1}{2}}\}$.
 	We choose $K$ so that
 	$$Y_0=\io\ut^+dz\leq cg^2(K).$$
 	By Lemma \ref{ynb}, we have
 	\begin{equation*}
 	\ut\leq K.
 	\end{equation*}
 	The proof is complete.
 \end{proof}
 \begin{clm}\label{pstb}We have
 	\begin{equation*}
 	\|\pst\|_\infty\leq c.
 	\end{equation*}
 	
 \end{clm}
 \begin{proof} Without loss of generality, we assume
 	\begin{equation*}
 	\|\pst\|_\infty=\|\pst^+\|_\infty.
 	\end{equation*}
 	Let $K, K_n$ be given as in \eqref{utbk}.
 	We use $(\pst-K_{n+1})^+$ as a test function in \eqref{ot1t} to obtain
 	\begin{equation}\label{ell3}
 	\io D_\tau(\ut)\nabla \pst\cdot\nabla (\pst-K_{n+1})^+dz+\tau\io \pst(\pst-K_{n+1})^+dz=\io (f-a\ut)(\pst-K_{n+1})^+dz.
 	\end{equation}
 	This together with \eqref{r22} implies
 	\begin{eqnarray}\label{psb}
 	\io\frac{1}{1+q\nuta}|\nabla (\pst-K_{n+1})^+|^2 dz&\leq&\io(f-a\ut)(\pst-K_{n+1})^+ dz\nonumber\\
 	&\leq &c\io(\pst-K_{n+1})^+ dz\equiv cZ_{n+1}	.
 	\end{eqnarray}
 	Remember that $2p>2$. We derive from the Sobolev embedding theorem and \eqref{psb} that
 	\begin{eqnarray}
 	\lefteqn{\left(\io\left[(\pst-K_{n+1})^+\right]^{2p}dz\right)^{\frac{1}{2p}}}\nonumber\\
 	&\leq&2\left(\io\left[(\pst-K_{n+1})^+-\frac{1}{|\Omega|}\io(\pst-K_{n+1})^+dz\right]^{2p}dz\right)^{\frac{1}{2p}}+cZ_{n+1}\nonumber\\
 	&\leq &c\left(\io|\nabla(\pst-K_{n+1})^+|^{\frac{2p}{p+1}}dz\right)^{\frac{p+1}{2p}}+cZ_{n+1}\nonumber\\
 	&\leq &c\left(\io(1+q\nuta)^{\frac{p}{p+1}}\frac{1}{(\io(1+q\nuta)^{\frac{p}{p+1}}}|\nabla(\pst-K_{n+1})^+|^{\frac{2p}{p+1}}dz\right)^{\frac{p+1}{2p}}+cZ_{n+1}\nonumber\\
 	&\leq &c\left(\io\frac{1}{1+q\nuta}|\nabla (\pst-K_{n+1})^+|^2 dz\right)^{\frac{1}{2}}\|1+q\nuta\|_{p}^{\frac{1}{2}}+cZ_{n+1}\nonumber\\
 	&\leq &cZ_{n+1}^{\frac{1}{2}}+cZ_{n+1}\leq cZ_{n+1}^{\frac{1}{2}} .\label{r111}
 	\end{eqnarray}
 	The last step is due to the fact that the sequence $\{Z_{n}\}$ is bounded.
 	As before, we have
 	\begin{equation}\label{r112}
 	Z_{n}=	\io(\pst-K_{n})^+dz\geq(K_{n+1}-K_n)|\{\pst\geq K_{n+1}\}|=\frac{K}{2^{n+1}}|\{\pst\geq K_{n+1}\}|.
 	\end{equation}
 	Using H\"{o}lder's inequality, \eqref{r111}, \eqref{r112}, and the fact that $\{Z_{n}\}$ is decreasing yields
 	\begin{eqnarray*}
 		Z_{n+1}&\leq &\left(\io\left[(\pst-K_{n+1})^+\right]^{2p}dz\right)^{\frac{1}{2p}}|\{\pst\geq K_{n+1}\}|^{1-\frac{1}{2p}}\nonumber\\
 		&\leq &cZ_{n+1}^{\frac{1}{2}}\left(\frac{2^{n+1}}{K}\right)^{1-\frac{1}{2p}}	Z_{n}^{1-\frac{1}{2p}}\nonumber\\
 		&\leq &\frac{c2^{(1-\frac{1}{2p})n}}{K^{1-\frac{1}{2p}}}Z_{n}^{1+\frac{p-1}{2p}}.
 	\end{eqnarray*}
 	By Lemma \ref{ynb}, if we choose $K$ so that
 	\begin{equation*}
 	Z_0=\io (\pst)^+dz\leq cK^\frac{2p-1}{p-1},
 	\end{equation*}
 	then we have
 	\begin{equation*}
 	\pst\leq K.
 	\end{equation*}
 	
 \end{proof}
 \begin{clm}\label{gup} The sequence $\{\nabla\ut\}$ is precompact in $W^{1,s}(\Omega)$ for each $s\in [1, p)$.
 \end{clm}
 \begin{proof} We need a version of Lemma \ref{lplap} for our approximation \cite{LI}. To this end,
 	we compute, for $\xi,\eta\in \mathbb{R}^2$, that
 	\begin{eqnarray}
 	\lefteqn{(|\eta|^2+\tau)^{\frac{p-2}{2}}\eta-	(|\xi|^2+\tau)^{\frac{p-2}{2}}\xi}\nonumber\\
 	&=&\int_0^1\frac{d}{dt}\left[(|\xi+t(\eta-\xi)|^2+\tau)^{\frac{p-2}{2}}(\xi+t(\eta-\xi))\right]dt\nonumber\\
 	&=&(\eta-\xi)\int_0^1(|\xi+t(\eta-\xi)|^2+\tau)^{\frac{p-2}{2}}dt\nonumber\\
 	&&+(p-2)\int_0^1(|\xi+t(\eta-\xi)|^2+\tau)^{\frac{p-4}{2}}(\eta-\xi)\cdot(\xi+t(\eta-\xi))(\xi+t(\eta-\xi))dt.
 	\end{eqnarray}
 	Subsequently,
 	\begin{eqnarray}
 	\lefteqn{\left((|\eta|^2+\tau)^{\frac{p-2}{2}}\eta-	(|\xi|^2+\tau)^{\frac{p-2}{2}}\xi\right)\cdot(\eta-\xi)}\nonumber\\
 	&=&|\eta-\xi|^2\int_0^1(|\xi+t(\eta-\xi)|^2+\tau)^{\frac{p-2}{2}}dt\nonumber\\
 	&&+(p-2)\int_0^1(|\xi+t(\eta-\xi)|^2+\tau)^{\frac{p-4}{2}}\left[(\eta-\xi)\cdot(\xi+t(\eta-\xi))\right]^2dt.\label{rev1}
 	\end{eqnarray}
 	The last integral has the estimate
 	\begin{eqnarray}
 	0&\leq &\int_0^1(|\xi+t(\eta-\xi)|^2+\tau)^{\frac{p-4}{2}}\left[(\eta-\xi)\cdot(\xi+t(\eta-\xi))\right]^2dt\nonumber\\
 	&\leq &|\eta-\xi|^2\int_0^1(|\xi+t(\eta-\xi)|^2+\tau)^{\frac{p-2}{2}}dt.
 	\end{eqnarray}
 	Use this in \eqref{rev1} and keep in mind the fact that $\tau\in (0,1), p\in (1,2)$ to derive
 	\begin{eqnarray}
 	\lefteqn{\left((|\eta|^2+\tau)^{\frac{p-2}{2}}\eta-	(|\xi|^2+\tau)^{\frac{p-2}{2}}\xi\right)\cdot(\eta-\xi)}\nonumber\\
 	&\geq &(p-1)|\eta-\xi|^2\int_0^1(|\xi+t(\eta-\xi)|^2+\tau)^{\frac{p-2}{2}}dt
 	\nonumber\\
 	&\geq&(p-1)\left(\int_0^1(|\xi+t(\eta-\xi)|^2+\tau)dt\right)^{\frac{p-2}{2}}|\eta-\xi|^2\nonumber\\
 	&\geq&(p-1)\left(1+|\xi|^2+|\eta|^2\right)^{\frac{p-2}{2}}|\eta-\xi|^2.\label{rev2}
 	\end{eqnarray}
 	Here we have used the fact that the function $s^{\frac{p-2}{2}}$ is convex. 
 	Obviously, we only have
 	\begin{equation}
 	\left((|\eta|^2+\tau)^{-\frac{1}{2}}\eta-	(|\xi|^2+\tau)^{-\frac{1}{2}}\xi\right)\cdot(\eta-\xi)\geq 0\ \ \mbox{for $\xi,\eta\in \mathbb{R}^2$.}\label{rev3}
 	\end{equation}
 	In view of \eqref{ue} and Claim \ref{utb}, we may assume
 	\begin{equation}\label{ucon}
 	u_\tau\rightarrow u\ \ \mbox{weakly in $W^{1,p}(\Omega)$ and strongly in $L^q(\Omega)$ for each $q\geq 1$ as $\tau\rightarrow 0$}
 	\end{equation}
 	(pass to a subsequence if necessary). From here on, the limit is always taken as $\tau\rightarrow 0$ (possibly) along a suitable subsequence.
 	It follows from \eqref{rev2} and \eqref{rev3} that
 	\begin{eqnarray}
 	\lefteqn{\left(\rho_{\tau}(|\nabla u_{\tau}|^2)\nabla u_{\tau}-\rho_{\tau}(|\nabla u|^2)	\nabla u\right)\cdot(\nabla u_{\tau}-\nabla u)}\nonumber\\
 	&\geq&(p-1)\left(1+|\nabla u_{\tau}|^2+|\nabla u|^2\right)^{\frac{p-2}{2}}|\nabla u_{\tau}-\nabla u|^2.\label{rev4}
 	\end{eqnarray}
 	Note that
 	\begin{equation}\label{r3}
 	\rho_{\tau}(|\nabla u|^2)\nabla u=(|\nabla u|^2+\tau)^{\frac{p-2}{2}}\nabla u+\beta(|\nabla u|^2+\tau)^{-\frac{1}{2}}\nabla u=0
 	\ \ \mbox{on $\{|\nabla u|=0\}$}.
 	\end{equation}
 	If we define $\rho(|\nabla u|^2)\nabla u$ to be $0$ on the set $\{|\nabla u|=0\}$, then we have
 	\begin{equation}\label{rcon}
 	\rho_{\tau}(|\nabla u|^2)\nabla u\rightarrow\rho(|\nabla u|^2)\nabla u\ \ \mbox{stronly in
 		$\left(W^{1, p^\prime}(\Omega)\right)^2$}.
 	\end{equation}
 	To see this, we first verify  that
 	\begin{eqnarray}
 	\rho_{\tau}(|\nabla u|^2)\nabla u
 	&\rightarrow &	|\nabla u|^{p-2}\nabla u+\beta	|\nabla u|^{-1}\nabla u=	\rho(|\nabla u|^2)\nabla u\ \ \mbox{a.e. on $\Omega$}.
 	\end{eqnarray}
 	Then check
 	\begin{equation}
 	\left|\rho_{\tau}(|\nabla u|^2)\nabla u\right|\leq |\nabla u|^{p-1}+\beta.
 	\end{equation}
 	Hence \eqref{rcon} follows from the Dominated Convergence Theorem.
 	
 	Keeping \eqref{rcon} in mind, we derive from \eqref{rev4} that
 	\begin{eqnarray}
 	\lefteqn{\io\left(\rho_{\tau}(|\nabla u_{\tau}|^2)\nabla u_{\tau}-\rho(|\nabla u|^2)	\nabla u\right)\cdot(\nabla u_{\tau}-\nabla u)dz}\nonumber\\
 	&=&\io\left(\rho_{\tau}(|\nabla u_{\tau}|^2)\nabla u_{\tau}-\rho_{\tau}(|\nabla u|^2)	\nabla u\right)\cdot(\nabla u_{\tau}-\nabla u)dz\nonumber\\
 	&& +\io\left(\rho_{\tau}(|\nabla u|^2)\nabla u-\rho(|\nabla u|^2)	\nabla u\right)\cdot(\nabla u_{\tau}-\nabla u)dz\nonumber\\
 	&\geq&(p-1)\io\left(1+|\nabla u_{\tau}|^2+|\nabla u|^2\right)^{\frac{p-2}{2}}|\nabla u_{\tau}-\nabla u|^2dz\nonumber\\
 	&& +\io\left(\rho_{\tau}(|\nabla u|^2)\nabla u-\rho(|\nabla u|^2)	\nabla u\right)\cdot(\nabla u_{\tau}-\nabla u)dz.\label{rev5}
 	\end{eqnarray}
 	On the other hand, we have from \eqref{ot3t} that
 	\begin{eqnarray}
 	\lefteqn{\io\left(\rho_{\tau}(|\nabla u_{\tau}|^2)\nabla u_{\tau}-\rho(|\nabla u|^2)	\nabla u\right)\cdot(\nabla u_{\tau}-\nabla u)dz}\nonumber\\
 	&=&\io( v_{\tau}-\tau|u_{\tau}|^{p-2}u_{\tau})( u_{\tau}- u)dz-\io\rho(|\nabla u|^2)	\nabla u\cdot(\nabla u_{\tau}-\nabla u)dz\nonumber\\
 	&\rightarrow & 0.
 	\end{eqnarray}
 	The last step is due to Claim \ref{pstb}, \eqref{ucon}, and the fact that $\left|\rho(|\nabla u|^2)	\nabla u\right|\in L^{p^\prime}(\Omega)$.
 	This together with \eqref{rev5} and \eqref{rcon} implies
 	\begin{equation}
 	\io\left(1+|\nabla u_{\tau}|^2+|\nabla u|^2\right)^{\frac{p-2}{2}}|\nabla u_{\tau}-\nabla u|^2dz\rightarrow 0.
 	\end{equation}
 	Remember that $2-p< p$. We estimate
 	\begin{eqnarray}
 	\lefteqn{	\io|\nabla u_{\tau}-\nabla u|dz}\nonumber\\
 	&=&\io\left(1+|\nabla u_{\tau}|^2+|\nabla u|^2\right)^{\frac{2-p}{4}}\left(1+|\nabla u_{\tau}|^2+|\nabla u|^2\right)^{\frac{p-2}{4}}|\nabla u_{\tau}-\nabla u|dz\nonumber\\
 	&\leq &\left(\io\left(1+|\nabla u_{\tau}|^2+|\nabla u|^2\right)^{\frac{2-p}{2}}dz\right)^{\frac{1}{2}}\nonumber\\
 	&&\cdot\left(\io\left(1+|\nabla u_{\tau}|^2+|\nabla u|^2\right)^{\frac{p-2}{2}}|\nabla u_{\tau}-\nabla u|^2dz\right)^{\frac{1}{2}}\nonumber\\
 	&\leq &c\left(\io\left(1+|\nabla u_{\tau}|^2+|\nabla u|^2\right)^{\frac{p-2}{2}}|\nabla u_{\tau}-\nabla u|^2dz\right)^{\frac{1}{2}}\nonumber\\
 	&\rightarrow& 0.\label{rev6}
 	\end{eqnarray} 
 	The claim follows.

 \end{proof}
 To finish the existence part of Theorem \ref{th1.1}, we conclude from
 Claims \ref{gup} and \ref{pstb} and \eqref{ptb}  that
 \begin{eqnarray}
 \nabla\ut&\rightarrow & \nabla u\ \ \ \mbox{ strongly in $\left(L^{s}(\Omega)\right)^2$ for each $p>s\geq 1$ and a.e. on $\Omega$,}\label{nuae}\\
 \pst&\rightarrow&  v\ \ \mbox{weakly in $W^{1,\frac{2p}{p+1}}(\Omega)$ and strongly in $L^q(\Omega)$ for each $q\geq 1$}.\nonumber
 \end{eqnarray} 
 Recall that
 \begin{equation*}
 D_\tau(\nabla \ut)=
 \left(\begin{array}{cc}
 1+\tau+\frac{(\ut)_x^2}{\nuta^2+\tau}\left(\frac{1}{1+q\nuta}-1\right)&\frac{(\ut)_x(\ut)_y}{\nuta^2+\tau}\left(\frac{1}{1+q\nuta}-1\right)\\
 \frac{(\ut)_x(\ut)_y}{\nuta^2+\tau}\left(\frac{1}{1+q\nuta}-1\right)&1+\tau+\frac{(\ut)_y^2}{\nuta^2+\tau}\left(\frac{1}{1+q\nuta}-1\right)
 \end{array}\right).
 \end{equation*} 
 Obviously, for a.e. $z\in \Omega$ we have that
 \begin{equation*}
 D_\tau(\nabla \ut(z))\rightarrow\left\{\begin{array}{cc}
 D(\nabla u(z))&\mbox{if $\nabla u(z)\ne 0$,}\\
 I&\mbox{if $\nabla u(z)= 0$}.
 \end{array}\right.
 \end{equation*}
 That is, each entry of $	D_\tau(\nabla \ut)$ converges a.e on $\Omega$. It is also bounded. Therefore,
 \begin{equation*}
 D_\tau(\nabla \ut)\nabla\pst\rightarrow D(\nabla u(z))\nabla v\ \ \ \mbox{weakly in $\left(L^{\frac{2p}{p+1}}(\Omega)\right)^{2}$}.
 \end{equation*}
 We may assume that
 \begin{equation*}
 \frac{\nabla\ut}{(\nuta^2+\tau)^{\frac{1}{2}}}\rightarrow h\ \ \mbox{weak$^*$ in $\left(L^\infty(\Omega)\right)^2$}.
 \end{equation*}
 We claim
 \begin{equation}\label{haha1}
 h(z)\in \partial\Phi(\nabla u(z))\ \ \ \mbox{for a.e. $z\in \Omega$,}
 \end{equation}
 where	$\Phi$ is given as in \eqref{fpd}.
 To see this, we derive from \eqref{nuae} that
 \begin{equation*}
 \frac{\nabla\ut(z)}{(|\nabla\ut(z)|^2+\tau)^{\frac{1}{2}}}\rightarrow \frac{\nabla u(z)}{|\nabla u(z)|} =h(z)\ \ \mbox{for a.e. $z\in \{|\nabla u|>0\}$}.
 \end{equation*}
 We always have
 \begin{equation*}
 \left|\frac{\nabla\ut}{(\nuta^2+\tau)^{\frac{1}{2}}}\right|\leq 1.
 \end{equation*}
 Consequently, $|h|\leq 1 $. This gives \eqref{haha1}
 
 We are ready to pass to the limit in \eqref{ot1t}-\eqref{ot2t}. This completes the proof of the existence part of Theorem \ref{th1.1}. 
 
 Finally, we remark that if $p\geq 2$ we need to apply a suitable version of (i) in Lemma \ref{lplap} in the proof of Claim \ref{gup}.
 We shall omit the details.
 \section{Regularity}\label{sec4}
 
 In this section we assume that condition \eqref{r1111} holds.
 
 \begin{clm}\label{nub}
 	Let	\eqref{r1111} be satisfied. Then $\{\ut\}$ is bounded in $W^{1,\infty}(\Omega)$.
 \end{clm}
 \begin{proof} We are inspired by an idea from (\cite{GT}, p. 300). For simplicity, we assume
 	\begin{equation}
 	\beta=1.
 	\end{equation}
 	First we can easily infer from the proof in \cite{AZ} that for each fixed $\tau>0$ the function $\ut$ is a $C^{1,\alpha}$ function for some $\alpha\in (0,1)$ (also see Chapters IV and V of \cite{LU}). Write \eqref{ot3t} in the form
 	\begin{eqnarray*}
 		\lefteqn{	-\left(|\nabla\ut|^2+\tau\right)^{\frac{p-2}{2}}\left(I+(p-2)\frac{\nabla\ut\otimes\nabla\ut}{|\nabla\ut|^2+\tau}\right):\nabla^2\ut}\\
 		&&-\left(|\nabla\ut|^2+\tau\right)^{-\frac{1}{2}}\left(I-\frac{\nabla\ut\otimes\nabla\ut}{|\nabla\ut|^2+\tau}\right):\nabla^2\ut=v_\tau-\tau|\ut|^{p-2}\ut\ \ \mbox{in $\Omega$.}
 	\end{eqnarray*} 
 	This puts us in a position to apply a result in \cite{CFL}, from whence follows that $\ut$ lies in $W^{2, q}(\Omega)$ for each $q>1$.
 	Thus we can calculate
 	\begin{eqnarray}
 	\mbox{div}\left((\nuta^2+\tau)^{\frac{p-2}{2}}\nabla\ut\right)&=&(\nuta^2+\tau)^{\frac{p-2}{2}}\Delta\ut\nonumber\\
 	&&+(p-2)(\nuta^2+\tau)^{\frac{p-4}{2}}\nabla\ut\otimes\nabla\ut :\nabla^2\ut\nonumber\\
 	&=&(\nuta^2+\tau)^{\frac{p-2}{2}}\left(I+\frac{p-2}{\nuta^2+\tau}\nabla\ut\otimes\nabla\ut\right):\nabla^2\ut\nonumber\\
 	&=&(\nuta^2+\tau)^{\frac{p-2}{2}}\left(E_p(\ut)_{xx}+2F_p(\ut)_{xy}+G_p(\ut)_{yy}\right),\label{dup}
 	\end{eqnarray}
 	where
 	\begin{eqnarray}
 	E_p&=&1+\frac{(p-2)(\ut)_x^2}{\nuta^2+\tau},\label{p11}\\
 	F_p&=&\frac{(p-2)(\ut)_x(\ut)_y}{\nuta^2+\tau},\label{p12}\\
 	G_p&=&1+\frac{(p-2)(\ut)_y^2}{\nuta^2+\tau}.\label{p13}
 	\end{eqnarray}
 	The preceding calculations are still valid if $p=1$. Subsequently,
 	\begin{equation}
 	\mbox{div}\left((\nuta^2+\tau)^{-\frac{1}{2}}\nabla\ut\right)=(\nuta^2+\tau)^{-\frac{1}{2}}\left(E_1(\ut)_{xx}+2F_1(\ut)_{xy}+G_1(\ut)_{yy}\right).\label{dul}
 	\end{equation}
 	Substitute \eqref{dup} and \eqref{dul} into \eqref{ot3t} and divide through the resulting equation by the coefficient of $(\ut)_{yy}$, which is
 	\begin{equation*}
 	H_p\equiv(\nuta^2+\tau)^{\frac{p-2}{2}}G_p+(\nuta^2+\tau)^{-\frac{1}{2}}G_1,
 	\end{equation*}  to deduce
 	\begin{eqnarray}
 	\lefteqn{-\frac{(\nuta^2+\tau)^{\frac{p-2}{2}}E_p+(\nuta^2+\tau)^{-\frac{1}{2}}E_1}{H_p}(\ut)_{xx}}\nonumber\\
 	&&-2\frac{(\nuta^2+\tau)^{\frac{p-2}{2}}F_p+(\nuta^2+\tau)^{-\frac{1}{2}}F_1}{H_p}(\ut)_{xy}\nonumber\\
 	&&-(\ut)_{yy}=\frac{\pst-\tau|\ut|^{p-2}\ut}{H_p}\equiv f_\tau.\label{plap1}
 	\end{eqnarray}
 	Let
 	\begin{equation}
 	w=(\ut)_x.
 	\end{equation}
 	By differentiating \eqref{plap1} with respect to $x$, we arrive at
 	\begin{equation}\label{eqfw}
 	-\mbox{div}\left(A_\tau\nabla w\right)=(f_\tau)_x,
 	\end{equation}
 	where 
 	\begin{equation}
 	A_\tau=\left(\begin{array}{cc}
 	\frac{(\nuta^2+\tau)^{\frac{p-2}{2}}E_p+(\nuta^2+\tau)^{-\frac{1}{2}}E_1}{H_p}& 2\frac{(\nuta^2+\tau)^{\frac{p-2}{2}}F_p+(\nuta^2+\tau)^{-\frac{1}{2}}F_1}{H_p}\\
 	0 & 1
 	\end{array}\right).
 	\end{equation}
 	For $\xi=(\xi_1,\xi_2)^T$,	we compute
 	\begin{eqnarray}
 	A_\tau\xi\cdot\xi&=&\frac{(\nuta^2+\tau)^{\frac{p-2}{2}}E_p+(\nuta^2+\tau)^{-\frac{1}{2}}E_1}{H_p}\xi_1^2\nonumber\\
 	&&+2\frac{(\nuta^2+\tau)^{\frac{p-2}{2}}F_p+(\nuta^2+\tau)^{-\frac{1}{2}}F_1}{H_p}\xi_1\xi_2+\xi_2^2\nonumber\\
 	&=&\frac{(\nuta^2+\tau)^{\frac{p-2}{2}}}{H_p}\left(E_p\xi_1^2+2F_p\xi_1\xi_2+G_p\xi_2^2\right)\nonumber\\
 	&&+	\frac{(\nuta^2+\tau)^{-\frac{1}{2}}}{H_p}\left(E_1\xi_1^2+2F_1\xi_1\xi_2+G_1\xi_2^2\right).\label{wine1}
 	\end{eqnarray}
 	We derive from \eqref{p11}-\eqref{p13} that
 	\begin{eqnarray}
 	E_p\xi_1^2+2F_p\xi_1\xi_2+G_p\xi_2^2&=&\left(I+\frac{p-2}{\nuta^2+\tau}\nabla\ut\otimes\nabla\ut\right)\xi\cdot\xi\nonumber\\
 	&=&|\xi|^2+(p-2)\frac{(\xi\cdot\nabla\ut)^2}{\nuta^2+\tau}.
 	\end{eqnarray}
 	Therefore,
 	\begin{eqnarray}
 	(p-1)|\xi|^2&\leq& E_p\xi_1^2+2F_p\xi_1\xi_2+G_p\xi_2^2\leq |\xi|^2\ \ \mbox{for $p\in (1,2)$}.
 	\end{eqnarray}
 	Similarly,
 	\begin{equation}
 	0\leq E_1\xi_1^2+2F_1\xi_1\xi_2+G_1\xi_2^2\leq |\xi|^2.\label{wine3}
 	\end{equation}
 	It is also easy to check
 	\begin{equation}
 	(p-1)\leq G_p=1+\frac{(p-2)(\ut)_y^2}{\nuta^2+\tau}\leq 1,
 	\end{equation}
 	while
 	\begin{equation}
 	0\leq G_1=1-\frac{(\ut)_y^2}{\nuta^2+\tau}=\frac{(\ut)_x^2+\tau}{\nuta^2+\tau}\leq 1.
 	\end{equation}
 	Now we consider	
 	\begin{eqnarray}
 	\frac{(\nuta^2+\tau)^{\frac{p-2}{2}}}{H_p}&=&\frac{(\nuta^2+\tau)^{\frac{p-2}{2}}}{	(\nuta^2+\tau)^{\frac{p-2}{2}}G_p+(\nuta^2+\tau)^{-\frac{1}{2}}G_1}\nonumber\\
 	&\leq &\frac{1}{G_p}
 	\leq \frac{1}{(p-1)}.\label{wine4}
 	\end{eqnarray}
 	On the other hand,
 	\begin{eqnarray}
 	\frac{(\nuta^2+\tau)^{\frac{p-2}{2}}}{H_p}&=&\frac{1}{G_p+(\nuta^2+\tau)^{-\frac{p-1}{2}}G_1}\nonumber\\
 	&=&\frac{1}{G_p+(\nuta^2+\tau)^{-\frac{p+1}{2}}((\ut)_x^2+\tau)}\nonumber\\
 	&\geq&\frac{1}{1+(\nuta^2+\tau)^{-\frac{p-1}{2}}}\nonumber\\
 	&=&\frac{(\nuta^2+\tau)^{\frac{p-1}{2}}}{1+(\nuta^2+\tau)^{\frac{p-1}{2}}}.\label{wine5}
 	\end{eqnarray}
 	We still need to estimate
 	\begin{eqnarray}
 	\frac{(\nuta^2+\tau)^{-\frac{1}{2}}}{H_p}&=&\frac{(\nuta^2+\tau)^{-\frac{1}{2}}}{	(\nuta^2+\tau)^{\frac{p-2}{2}}G_p+(\nuta^2+\tau)^{-\frac{1}{2}}G_1}\nonumber\\
 	&\leq &	\frac{(\nuta^2+\tau)^{-\frac{p-1}{2}}}{G_p}\leq \frac{1}{(p-1)(\nuta^2+\tau)^{\frac{p-1}{2}}}\label{wine6}
 	\end{eqnarray}
 	and
 	\begin{eqnarray}
 	\frac{1}{H_p}&=&\frac{1}{	(\nuta^2+\tau)^{\frac{p-2}{2}}G_p+(\nuta^2+\tau)^{-\frac{1}{2}}G_1}\nonumber\\
 	&\leq &\frac{1}{(p-1) (\nuta^2+\tau)^{\frac{p-2}{2}}}.\label{wine7}
 	\end{eqnarray}
 	We are now ready to show that 
 	\begin{equation*}
 	w\in L^\infty_{\mbox{loc}}(\Omega).
 	\end{equation*}	To this end, we fix a point $z_0\in \Omega$. Then pick a number $R$ from $(0, \min\{\mbox{dist}(z_0,\po),1\})$. Define a sequence of concentric balls $B_{R_n}(z_0)$ in $\Omega$ as follows:
 	\begin{equation*}
 	B_{R_n}(z_0)=\{z:|z-z_0|<R_n\},
 	\end{equation*}
 	where
 	\begin{equation*}
 	R_n=\frac{R}{2}+\frac{R}{2^{n+1}},\ \ \ n=0,1,2,\cdots.
 	\end{equation*} Choose a sequence of smooth functions $\theta_n$ so that
 	\begin{eqnarray*}
 		\theta_n(z)&=& 1 \ \ \mbox{in $B_{R_n}(z_0)$},\\
 		\theta_n(z)&=&0\ \ \mbox{outside $B_{R_{n-1}}(z_0)$},\\
 		|\nabla \theta_n(z)|&\leq & \frac{c2^n}{R}\ \ \mbox{for each $z\in \mathbb{R}^2$,}\ \ \ \mbox{and}\\
 		0&\leq &\theta_n(z)\leq 1\ \ \mbox{in $\mathbb{R}^2$.}
 	\end{eqnarray*}
 	Select
 	\begin{equation*}
 	K\geq 2
 	\end{equation*}
 	as below.
 	Set
 	\begin{equation*}
 	K_n=K-\frac{K}{2^{n+1}},\ \ \ n=0,1,2,\cdots.
 	\end{equation*}
 	Obviously,
 	\begin{equation}\label{kn}
 	K_n\geq 1\ \ \ \mbox{for all $n$.}
 	\end{equation}
 	Without loss of generality,  assume that
 	\begin{equation}\label{kn1}
 	\max_{B_{\frac{R}{2}}(z_0)}w=\max_{B_{\frac{R}{2}}(z_0)}w^+.
 	\end{equation}
 	We use $\theta_{n+1}^2(w-K_{n+1})^+$ as a test function in \eqref{eqfw} to obtain
 	\begin{eqnarray}
 	\io A_\tau\nabla w\cdot\nabla(w-K_{n+1})^+\theta_{n+1}^2dz&=&-2\io A_\tau\nabla w\cdot\nabla\theta_{n+1}(w-K_{n+1})^+\theta_{n+1}dz\nonumber\\
 	&&-\io f_\tau\theta_{n+1}^2(w-K_{n+1})^+_xdz\nonumber\\
 	&&-2\io f_\tau\theta_{n+1}\left(\theta_{n+1}\right)_x(w-K_{n+1})^+dz.\label{eqfw1}
 	\end{eqnarray}
 	We deduce from \eqref{wine1}, \eqref{wine3}, \eqref{wine4}, and \eqref{wine5} that
 	\begin{eqnarray}
 	A_\tau\nabla w\nabla(w-K_{n+1})^+&=&A_\tau\nabla(w-K_{n+1})^+\cdot \nabla(w-K_{n+1})^+\nonumber\\
 	&\geq&\frac{c(\nuta^2+\tau)^{\frac{p-1}{2}}}{1+(\nuta^2+\tau)^{\frac{p-1}{2}}}|\nabla(w-K_{n+1})^+|^2.
 	\end{eqnarray}
 	Observe from \eqref{kn} that
 	\begin{equation*}
 	\nuta\geq |(\ut)_x|=(\ut)_x=w\geq K_{n+1}\geq 1\ \ \mbox{on the set $\{w\geq K_{n+1}\}$}.
 	\end{equation*}
 	Therefore,
 	\begin{eqnarray*}
 		A_\tau\nabla w\nabla(w-K_{n+1})^+&\geq &\frac{c(1+\tau)^{\frac{p-1}{2}}}{1+(1+\tau)^{\frac{p-1}{2}}}|\nabla(w-K_{n+1})^+|^2\nonumber\\
 		&\geq &c|\nabla(w-K_{n+1})^+|^2.
 	\end{eqnarray*}
 	In view of \eqref{wine3} and \eqref{wine6}, we obtain
 	\begin{eqnarray*}
 		\left|A_\tau\nabla w\right|(w-K_{n+1})^+&\leq&  \left(c+\frac{1}{(p-1)(\nuta^2+\tau)^{\frac{p-1}{2}}}\right)|\nabla w|(w-K_{n+1})^+\nonumber\\
 		&\leq &\left(c+\frac{1}{(p-1)(1+\tau)^{\frac{p-1}{2}}}\right)|\nabla (w-K_{n+1})^+|(w-K_{n+1})^+\nonumber\\
 		&\leq& c|\nabla (w-K_{n+1})^+|(w-K_{n+1})^+.
 	\end{eqnarray*}
 	Plug the preceding two inequalities into \eqref{eqfw1} and apply Young's inequality
 	appropriately to derive
 	\begin{eqnarray*}
 		\int_{B_{R_n}(z_0)}|\nabla(w-K_{n+1})^+|^2\theta_{n+1}^2dz&\leq& \frac{c4^n}{R^2}\int_{B_{R_n}(z_0)}|(w-K_{n+1})^+|^2dz\nonumber\\
 		&&+c\int_{	S_{n+1}}f_\tau^2dz,
 	\end{eqnarray*}
 	where
 	\begin{equation*}
 	S_{n+1}=B_{R_n}(z_0)\cap\{w\geq K_{n+1}\}.
 	\end{equation*}
 	Set 
 	\begin{equation*}
 	Y_n=\int_{B_{R_n}(z_0)}|(w-K_{n})^+|^2dz.
 	\end{equation*}
 	For each $s>2$ we conclude from Poincar\'{e}'s inequality that
 	\begin{eqnarray}
 	\left(\int_{B_{R_n}(z_0)}|(w-K_{n+1})^+\theta_{n+1}|^sdz\right)^{\frac{2}{s}}&\leq &c\left(\int_{B_{R_n}(z_0)}|\nabla\left((w-K_{n+1})^+\theta_{n+1}\right)|^{\frac{2s}{s+2}}dz\right)^{\frac{s+2}{s}}\nonumber\\
 	&\leq &c\int_{B_{R_n}(z_0)}|\nabla\left((w-K_{n+1})^+\theta_{n+1}\right)|^{2}dz|S_{n+1}|^{\frac{2}{s}}\nonumber\\
 	&\leq &	\frac{c4^n}{R^2}Y_n|S_{n+1}|^{\frac{2}{s}}+c\int_{	S_{n+1}}f_\tau^2dz|S_{n+1}|^{\frac{2}{s}}.\label{wine8}
 	\end{eqnarray}
 	By H\"{o}lder's inequality with exponent $\frac{s}{2}$,
 	\begin{eqnarray}
 	Y_{n+1}
 	&\leq &\int_{B_{R_n}(z_0)}|(w-K_{n+1})^+\theta_{n+1}|^2dz\nonumber\\
 	&\leq &\left(\int_{B_{R_n}(z_0)}|(w-K_{n+1})^+\theta_{n+1}|^sdz\right)^{\frac{2}{s}}|S_{n+1}|^{1-\frac{2}{s}}\nonumber\\
 	&\leq &	\frac{c4^n}{R^2}Y_n|S_{n+1}|+c\int_{	S_{n+1}}f_\tau^2dz|S_{n+1}|.\label{wine10}
 	\end{eqnarray}
 	%
 	Our assumption on $p$ in \eqref{r1111} implies
 	\begin{equation*}
 	2(2-p)<p.
 	\end{equation*}
 	With this in mind, we estimate the last integral in \eqref{wine10} from \eqref{wine7} as follows:
 	\begin{eqnarray}
 	\int_{S_{n+1}}f_\tau^2dz&\leq&\int_{S_{n+1}}
 	\frac{(\pst-\tau|\ut|^{p-2}\ut)^2}{\left((p-1) (\nuta^2+\tau)^{\frac{p-2}{2}}\right)^2}dz\nonumber\\
 	&\leq &
 	c\int_{S_{n+1}}(\nuta^2+\tau)^{2-p}dz\nonumber\\
 	&\leq&c\left(\int_{S_{n+1}}(\nuta^2+\tau)^{\frac{p}{2}}dz\right)^{\frac{2(2-p)}{p}}|S_{n+1}|^{\frac{3p-4}{p}}\leq c|S_{n+1}|^{\frac{3p-4}{p}},\label{winer}
 	\end{eqnarray}
 	where Claims \ref{utb} and \ref{pstb} have been used.
 	Note that
 	\begin{eqnarray}
 	Y_n&=&\int_{B_{R_n}(z_0)}|(w-K_{n})^+|^2dz\nonumber\\
 	&\geq &\int_{	S_{n+1}}(K_{n+1}-K_n)^2dz=|S_{n+1}|\frac{K^2}{4^{n+2}},\ \ \mbox{and}\label{wine11}\\
 	|S_{n+1}|&=&|S_{n+1}|^{\frac{2(2-p)}{p}}|S_{n+1}|^{\frac{3p-4}{p}}\leq cR^{\frac{4(2-p)}{p}}|S_{n+1}|^{\frac{3p-4}{p}}.\nonumber
 	\end{eqnarray}
 	Combining the preceding four inequalities
 	yields
 	\begin{eqnarray*}
 		Y_{n+1}&\leq& \frac{c4^n4^{\frac{n(3p-4)}{p}}}{R^{\frac{6p-8}{p}}K^{\frac{6p-8}{p}}}Y_n^{1+\frac{3p-4}{p}}+c\frac{4^{\frac{4n(p-1)}{p}}}{K^{\frac{8(p-1)}{p}}}Y_n^{1+\frac{3p-4}{p}}\nonumber\\
 		&\leq & \frac{c16^n}{R^{\frac{6p-8}{p}}K^{\frac{6p-8}{p}}}Y_n^{1+\frac{3p-4}{p}}.
 	\end{eqnarray*}
 	Choose $K\geq 2$ so that
 	\begin{eqnarray*}
 		Y_0\leq c\left(R^{\frac{6p-8}{p}}K^{\frac{6p-8}{p}}\right)^{\frac{p}{3p-4}}=cR^2K^2,
 	\end{eqnarray*}
 	from whence follows
 	\begin{equation*}
 	K\geq c\left(\frac{1}{R^2}Y_0\right)^{\frac{1}{2}}.
 	\end{equation*}
 	This combined with Lemma \ref{ynb} and \eqref{kn1} gives
 	\begin{eqnarray}
 	\sup_{B_{\frac{R}{2}}(z_0)}|(\ut)_x|&\leq& 2+c\left(\frac{1}{R^2}\int_{B_{R}(z_0)}(\ut)_x^2dz\right)^{\frac{1}{2}}.\label{fr}
 	\end{eqnarray}
 	Obviously, the above estimate remains valid if we substitute $(\ut)_y$ with $(\ut)_x$. This leads to the following inequality
 	\begin{equation*}
 	\sup_{B_{\frac{R}{2}}(z_0)}|\nabla\ut|\leq c+c\left(\frac{1}{R^2}\int_{B_{R}(z_0)}\nuta^2dz\right)^{\frac{1}{2}}.
 	\end{equation*}
 	This is the so-called local interior estimate. However, it is not difficult to extend it to an $L^\infty(\Omega)$ estimate (\cite{GT}, p. 303). Indeed, if $z_0\in \po$, our assumption on the boundary implies that there exist a neighborhood $U(z_0)$ of $z_0$ and a $C^{1,1}$ diffeomorphism $\mathbb{T}$ defined on $U(z_0)$ such that the image of $U(z_0)\cap\Omega$ under $\mathbb{T}$ is the half ball $B_\delta^+(0,\eta_2^0)=\{(\eta_1, \eta_2): \eta_1^2+(\eta_2-\eta_2^0)^2< \delta^2, \eta_1>0\} $, where $\delta>0,  (0, \eta_2^0)=\mathbb{T}(z_0)$. That is to say, we have straightened a portion of the boundary into a segment
 	of $\eta_1=0$ in the $(\eta_1, \eta_2)$ plane \cite{CFL}.
 	Set
 	\begin{equation*}
 	\tilde{u}=u\circ\mathbb{T}^{-1}, \ \ \tilde{w}=\tilde{u}_{\eta_1}.
 	\end{equation*}
 	Then we can choose $\mathbb{T} $ so that $\tilde{w}$  satisfies the boundary condition
 	\begin{equation}\label{bc}
 	\tilde{w}(0,\eta_2)=\tilde{u}_{\eta_1}(0,\eta_2)=0.
 	\end{equation}
 	One way of doing this is to pick $\mathbb{T}=\left(\begin{array}{c}
 	f_1(z)\\
 	f_2(z)
 	\end{array}\right)$ so that the graph of $f_1(z)=0$ is $U(z_0)\cap\po$ and the set of vectors $\{\nabla f_1, \nabla f_2\}$ is orthogonal.
 	By a result in \cite{X6}, $\tilde{w}$ satisfies the equation
 	\begin{equation*}
 	-\mbox{div}\left[(J_{\mathbb{T}}^TA_\tau J_{\mathbb{T}})\circ \mathbb{T}^{-1}\nabla\tilde{w}\right]=(h_1, h_2)(A_\tau J_{\mathbb{T}})\circ \mathbb{T}^{-1}\nabla\tilde{w}+(f_\tau)_x\circ \mathbb{T}^{-1}\ \ \mbox{in $B_\delta^+(0,\eta_0)$},
 	\end{equation*}
 	where $J_{\mathbb{T}}$ is the Jacobian matrix of $\mathbb{T}$, i.e., 
 	\begin{equation*}
 	J_{\mathbb{T}}=\nabla\mathbb{T}
 	\end{equation*} and the functions $h_1$ and $ h_2$ comprise first and second order partial derivatives of $\mathbb{T}$ and  are, therefore, bounded by our assumption 
 	on $\mathbb{T}$. In view of \eqref{bc}, the method employed to prove \eqref{fr} still works here. The only difference is that we use $B_{R_n}^+(0, \eta_2^0)$ instead of $B_{R_n}(0, \eta_2^0)$ in the proof.
 	Hence we can conclude that
 	\begin{eqnarray}
 	\sup_{\Omega}|\nabla\ut|\leq c+c\left(\int_{\Omega}\nuta^2dz\right)^{\frac{1}{2}}.\label{haha}
 	\end{eqnarray}
 	By \eqref{otn9}, for each $\varepsilon>0$ there is a number $c$ such that
 	\begin{equation*}
 	\|\nabla\ut\|_2\leq \varepsilon \|\nabla\ut\|_\infty+c\|\nabla\ut\|_1\leq \varepsilon \|\nabla\ut\|_\infty+c.
 	\end{equation*}
 	This together with \eqref{haha} gives the desired result.
 \end{proof}
 We would like to remark that it does not seem possible that one can derive equations similar to \eqref{eqfw} for the partial derivatives of $u_\tau$ when the space dimension $N\geq 3$. This is the main reason for the assumption $N=2$.
 We also give an indication why  the case where $p\geq 2$ is easier. This is mainly due to the fact that in this case $H_p$ is bounded away from $0$ below on the set $\{w\geq 1\}$ (see \eqref{wine7}).
 \begin{clm}
 	The sequence $\{\pst\}$ is precompact in $W^{1,2}(\Omega)$.
 \end{clm}
 \begin{proof} 
 	The preceding Claim combined with \eqref{rue} and \eqref{mell} implies that $\{\pst\}$ is a bounded sequence in $W^{1,2}(\Omega)$.
 	Recall that
 	each entry of $	D_\tau(\nabla \ut)$ is bounded and converges a.e on $\Omega$. Therefore,
 	\begin{equation*}
 	D_\tau(\nabla \ut)\nabla\pst\rightarrow D(\nabla u(z))\nabla v\ \ \ \mbox{weakly in $\left(L^2(\Omega)\right)^{2\times 2}$}.
 	\end{equation*}
 	Adding $\mbox{div}\left(D(\nabla u )\nabla v\right)$ to both sides of  \eqref{ot1t} yields
 	\begin{eqnarray*}
 		-\mbox{div}\left(D_\tau(\nabla u_{\tau} )(\nabla v_{\tau}-\nabla v)\right)&=&\mbox{div}\left(D_\tau(\nabla u_{\tau}) \nabla v\right)-au_\tau-\tau\pst+f.
 	\end{eqnarray*}
 	Use $\pst- v$ as a test function in the above equation and keep in mind Claim \ref{nub} to get
 	\begin{eqnarray*}
 		c\io|\nabla v_{\tau}-\nabla v|^2dx&\leq&\io D_\tau(\nabla u_{\tau} )(\nabla v_{\tau}-\nabla v)\cdot(\nabla v_{\tau}-\nabla v)dz\nonumber\\
 		&=&\io D_\tau(\nabla u_{\tau}) \nabla v(\nabla v_{\tau}-\nabla v)dz\nonumber\\
 		&&+\io(-au_\tau-\tau\pst+f)(\pst- v)dz\nonumber\\
 		&\rightarrow& 0.
 	\end{eqnarray*}
 	The proof is complete.
 \end{proof}

\end{document}